\documentclass[11pt,a4paper,twoside]{article}
\usepackage{amsmath, amssymb, amsthm} 

\def\calf{{\cal F}}
\def\eps{\varepsilon}
\def\RR{\mathbb{R}}

\newcommand\tr{\operatorname{Trace}}
\newcommand\Div{\operatorname{div}}
\newcommand\id{\operatorname{id}}
\def\Ric{\operatorname{Ric}}

\def\eq{\hspace*{-1.5mm}&=&\hspace*{-1.5mm}}

\def\dt{\partial_t}

\newtheorem{corollary}{Corollary}
\newtheorem{definition}{Definition}
\newtheorem{example}{Example}
\newtheorem{remark}{Remark}
\newtheorem{lemma}{Lemma}
\newtheorem{proposition}{Proposition}
\newtheorem{theorem}{Theorem}

\author{Vladimir Rovenski\footnote{Department of Mathematics, University of Haifa, Mount Carmel, 3498838 Haifa, Israel.
    \newline
    e-mail: {\tt vrovenski@univ.haifa.ac.il}
    }
    \ and \
    Robert Wolak\footnote{
    Faculty of Mathematics and Computer Science, Institute of Mathematics of Jagiellonian University,
    Lojasiewicza 6, 30-348 Krakow, Poland.
    e-mail: {\tt robert.wolak@im.uj.edu.pl}}
}

\title{The Partial Ricci Flow on $\mathfrak{g}$-foliations}

\begin{document}

\date{}

\maketitle

\begin{abstract}
In the paper we introduce new metric structures on $\mathfrak{g}$-foliations that are  less rigid than the well-known structures: almost contact and 3-Sasakian structures as well as  $f$-{structures with parallelizable kernel} and  almost para-$\phi$-structures with complemented frames. We~discuss the properties of the new structures in order to demonstrate  similarities with the corresponding classical structures.  Then using the flow of metrics on a $\mathfrak{g}$-foliation, we build deformation retraction of our structures with positive partial Ricci curvature 
onto the subspace of the aforementioned classical structures.

\vskip1.5mm\noindent
\textbf{Keywords}: ${\mathfrak{g}}$-foliation, totally geodesic, partial Ricci curvature,
almost contact structure, 3-Sasakian structure, $f$-structure, almost para-$\phi$-structure, flow of metrics.

\vskip1.5mm\noindent
\textbf{Mathematics Subject Classifications (2010)} Primary 53C12; Secondary 53C21
\end{abstract}

\section*{Introduction}

In \cite{A}, D.\,Alekseevsky and P.\,Michor have began the study of
general ${\mathfrak{g}}$-manifolds, that is
smooth manifolds $M$ with an action of constant rank of a Lie algebra ${\mathfrak{g}}$, or a bit more restrictively,
a locally free action.
Several classes of ${\mathfrak{g}}$-manifolds with
foliations defined by the action of ${\mathfrak{g}}$
and some additional geometrical structures have been studied in great detail.
If we take a 1-dimensional Lie algebra, a ${\mathfrak{g}}$-manifold is then just a smooth manifold with a nonzero vector field.
In~this case, we obtain almost contact, strict contact, (almost) contact metric structures as well as $K$-contact and Sasakian ones, cf. \cite{b2010}.
If~we take a higher dimensional abelian Lie algebra, then as examples of such ${\mathfrak{g}}$-manifolds can serve
${\mathcal K}, {\mathcal S}$ or ${\mathcal C}$ manifolds. 3-Sasakian manifolds, cf. \cite{b2010}, form a special class of $so(3)$-manifolds.
An~$f$-{structure with parallelizable kernel and an {almost para-$\phi$-structure with complemented frames}
can be used to model ${\mathfrak{g}}$-foliations of higher (co)dimension, e.g., \cite{diT06,fp-m,gy,od,yan},
these generalize well-known almost complex, the almost contact, the almost product and the almost para-contact structures.

A~foliation $\mathcal F$ on a Riemannian manifold $(M,g)$ is called a \textit{tangentially Lie foliation} if there is a complete Lie parallelism $\{\xi_1,\ldots,\xi_p\}$ along its leaves that preserves the horizontal subbundle ${\mathcal F}^{\perp}$, e.g., \cite{ca-toh}.
For~convenience, suppose that the vector fields $\xi_i$ are orthonormal.
Obviously, a manifold with a tangentially Lie foliation is a ${\mathfrak{g}}$-manifold, and a ${\mathfrak{g}}$-manifold with a locally free ${\mathfrak{g}}$-action admits a tangentially ${\mathfrak{g}}$-foliation.
Another context in which ${\mathfrak{g}}$-manifolds appear is the geometrical study of differential equations on
manifolds,
cf.~\cite{olv,Wo-comp}.


Due to the importance of $\mathfrak{g}$-foliations, we are going to generalize their particular classes mentioned above,
and we hope that these generalizations will also find many applications. Namely, in the work, we introduce new metric structures on $\mathfrak{g}$-foliations that are  less rigid than such well-known structures as almost contact and 3-Sasakian structures,
an $f$-{structure with parallelizable kernel and an {almost para-$\phi$-structure with complemented frames}.


We assume that the foliations under consideration are totally geodesic.
In~the paper, we discuss the properties of our structures in order to demonstrate their similarity with the corresponding classical structures. Since the classical structures have positive constant partial Ricci curvature, we pay a special attention to those structures with positive partial Ricci curvature.
Using the flow of metrics, which in the case of $\mathfrak{g}$-foliations reduces to ODE,
we build  deformation retractions of our structures with positive partial Ricci curvature into the subspace of the aforementioned classical structures. In particular, we prove that a weak Sasakian structure can be deformed to a Sasakian structure (Theorem~\ref{C-PRF-g-1}),
a weak $p$-Sasakian structure to a $p$-Sasakian structure (Theorem~\ref{T-PRF-g2}),
and similar results for a metric weak $f$-structure (Theorem~\ref{T-PRF-g2b}), and a metric weak para-$\phi$-structure
(Theorem~\ref{T-PRF-g2c}). The technical results are gathered and proved in the last section.

\section{Weak contact structures}
\label{sec:main}

Almost contact and contact metric structures have been a very important and lively field of research for some decades. In what follows we propose to weaken one of the conditions and we investigate the geometry of such weak structures. Then we study the geometry of manifolds admitting finite families of  weak contact structures, which are compatible in some sense. A curvature condition, namely that  partial Ricci curvature along the orthogonal distribution is positive, permits to define a deformation retraction on a very special weak contact metric structure -- a weak Sasakian structure.

\begin{definition}\label{D-weak_ac}\rm
A \textit{weak almost contact structure} is an odd-dimensional manifold $M$
endowed with a $(1,1)$-tensor field $\phi$, a characteristic vector field $\xi$, a dual 1-form $\eta$,
and a nonsingular $(1,1)$-tensor field $Q$ satisfying
\begin{eqnarray}\label{E-Q1}
 && \phi^2 = -Q +\eta\otimes\xi,\quad \eta(\xi)=1, \\
\label{E-Q0}
 && Q\,\xi=\xi.
\end{eqnarray}
 A~weak almost contact structure is said to be \textit{normal} if
 $N_\phi+2\,d\eta\otimes\xi=0$,
where $N_\phi$ is the \textit{Nijenhuis torsion} of a $(1,1)$-tensor~$\phi$,
\[
 N_\phi(X,Y):=\phi^2[X,Y]+[\phi X,\phi Y]-\phi[\phi X,Y]-\phi[X,\phi Y].
\]
If for a {weak almost contact structure} there is a metric $g$ such that
\begin{equation}\label{E-Q2}
 g(\phi X,\phi Y)=g(X,QY)-\eta(X)\,\eta(Y),
\end{equation}
then we get a \textit{weak almost contact metric structure}.
We get a \textit{weak contact metric structure}, if, in addition,
\begin{equation}\label{E-Qd}
 g(X,\phi Y)=d\eta(X,Y).
\end{equation}
A normal weak contact metric structure will be called a~\textit{weak Sasakian structure}.
\end{definition}

The following proposition generalizes \cite[Theorem~4.1]{b2010}.

\begin{proposition}
(a)~Suppose that $M^{2n+1}$ admits a {weak almost contact structure} $(\phi,\xi,\eta,Q)$.
 Then $\phi$ has rank $2n$ and the following equalities hold:
\[
 \phi\,\xi=0,\quad \eta\circ\phi=0,\quad [Q,\,\phi]=0.
\]
(b)~For a weak almost contact metric structure, the tensor field $\phi$ is skew-symmetric,
and the tensor field $Q$ is self-adjoint,
\begin{equation}\label{E-Q2-g}
 g(\phi X, Y) = -g(X, \phi Y),\quad g(QX,Y)=g(X,QY).
\end{equation}
\end{proposition}

\begin{proof}
(a)~By Definition~\ref{D-weak_ac}, $\phi^2\xi=0$,
hence either $\phi\,\xi=0$ or $\phi\,\xi$ is a nontrivial vector of $\ker\phi$.
Applying \eqref{E-Q1} to $\phi\,\xi$, we get $Q(\phi\,\xi)=\eta(\phi\,\xi)\xi$,
If $\phi\,\xi=\mu\,\xi$ for some nonzero $\mu:M\to\RR$ then
$0=\phi^2\xi =\mu\cdot\phi\,\xi=\mu^2\xi\ne0$ -- a contradiction.
Assuming $\phi\,\xi=\mu\,\xi+X$ for some $\mu:M\to\RR$ and nonzero $X\in\ker\eta$, again by \eqref{E-Q1} we get
$QX=0$ -- a contradiction.  Thus, $\phi\,\xi=0$.

Next, since $\phi\,\xi =0$ everywhere, ${\rm rank}\,\phi < 2n+1$.
If a vector field $\tilde\xi$ satisfies $\phi\,\tilde\xi=0$, then \eqref{E-Q1} gives
$Q\,\tilde\xi=\eta(\tilde\xi)\,\xi$. One may write $\tilde\xi=\mu\xi+X$ for some $\mu:M\to\RR$ and $X\in\ker\eta$.
This yields $Q X =0$, hence $X=0$ and $\tilde\xi$ is collinear with $\xi$, and so ${\rm rank}\,\phi = 2n$.
 To show $\eta\circ\phi=0$, observe that from \eqref{E-Q1} and \eqref{E-Q0}
we get $[Q,\,\phi]=0$.
 Since $\phi\,\xi=0$, we also have, applying \eqref{E-Q1},
\[
 \eta(\phi X)=\phi^3 X +Q(\phi X) = Q\phi X -\phi Q X = [Q,\phi](X) = 0
\]
for any $X$, that proves the claim.

(b)~Let us take $Y=\xi$, using the property \eqref{E-Q0}, the formula \eqref{E-Q2} yields  $\eta(X)=g(X,\xi)$.
By the same formula, the~tensor field $Q$ is self-adjoint. For any $Y\in{\cal D}$ there is $\tilde Y\in{\cal D}$ such that $\phi\,Y=\tilde Y$.
Thus, skew-symmetry of $\phi$ follows from \eqref{E-Q2}
for $X\in{\cal D}$ and $\tilde Y$.
\end{proof}

Notice that conditions \eqref{E-Q2-g} characterize in a sense compatible metrics $g$.

\begin{definition}\rm
We say that an endomorphism $F:{\cal D}\to{\cal D}$ has a \textit{skew-symmetric representation} if
there exists an isomorphism $J:{\cal D}\to {\cal D}^\ast$ such that the (0,2)-tensor $\tilde F: (X,Y) \to J(Y)(F(X))$ is skew-symmetric;
or equivalently, for any $x\in M$ there exist a neighborhood $U_x\subset M$ and a~frame $\{e_i\}$ on $U_x$,
for which $F$ has a skew-symmetric matrix.
\end{definition}

\begin{remark}\rm
Each of two conditions in the above definition is equivalent to the following:
\textit{there exist a sub-Riemannian metric on ${\cal D}$ and an isomorphism $J^\sharp:{\cal D}\to {\cal D}$
such that the $(1,1)$-tensor $F^* J^\sharp:{\cal D}\to{\cal D}$ is skew-symmetric}.
Here, $J^\sharp(X) = J(X)^\sharp\ (X\in{\cal D})$ is the link between $J^\sharp$ and $J$.
As a frame $\{e_i\}$ on $U_x$ one can take any orthonormal frame on $U_x$ of such a metric.
\end{remark}

In the case of weak structures to prove the existence of a compatible metric we need an additional condition.

\begin{proposition}\label{T-04}
If $\phi$ of a weak almost contact structure has a skew-symmetric representation
then this structure admits a compatible metric.
\end{proposition}

\begin{proof} Assume that \eqref{E-Q1} and \eqref{E-Q0} hold and $\phi$ has a skew-symmetric matrix in a local frame $\{e_i\}$ on a domain $U\subset M$.
There exists metric $g_U$ on $U$ such that $\{e_i\}$ is orthonormal. Thus, \eqref{E-Q2-g} holds for $g_U$,
in particular, $g_U(\phi(X),X)=0$ for all $X\in TU$.
The last property is preserved when summing a finite number of metrics.
Hence, using a partition of unity, we get a metric $g$ on $M$ with the same property $g(\phi(X),X)=0$ for all $X\in {\cal X}_M$,
i.e., $\phi$ is skew-symmetric for $g$. Thus, \eqref{E-Q2} is valid.
\end{proof}

\begin{example}\label{Ex-06}\rm
(a)~For $Q=\id$, where $\id : TM \to TM$ is the identity mapping, Definition~\ref{D-weak_ac} gives an almost contact (metric) structure.
For the sectional curvature of a Sasakian structure, $K(\xi,X)=1$ for $X\bot\,\xi$, see \cite[Theorem~7.2]{b2010}.

(b)~Let $(M,g,\phi,\xi,\eta)$ be an almost contact metric manifold.
For arbitrary $(1,1)$-tensor $\phi'$ commuting with $\phi$ on $M$,
define $\tilde\phi=\phi+\phi'$ and $Q=\id-(\phi\phi'+\phi'\phi)-(\phi')^2$ on ${\cal D}$.
Then $(M,\tilde\phi,\xi,\eta,Q)$ is a \textit{weak almost contact manifold} when $|\phi'|$ is sufficiently small.
\end{example}

If a weak Sasakian structure has positive mixed sectional curvature we can deform the structure to a Sasakian structure. In fact, a similar result can be proved for finite families of such structures, see Theorem 2,  therefore we just formulate the  result.

\begin{theorem}\label{C-PRF-g-1}
Let $(\phi,\xi,\eta,Q)$ be a weak Sasakian structure on a Riemannian manifold $(M,g_0)$ such that $\,\xi$ is a unit Killing vector. If the sectional curvature $K_0(\xi,X)>0$ for $X\bot\,\xi$ and metric~$g_0$, then there exists a smooth family of metrics $g_t\ (t\in\RR)$ that converges exponentially fast, as $t\to-\infty$, to a limit metric $\hat g$ giving a Sasakian structure on $M$ $($thus, $\widehat K(\xi,X)=1$ for $X\bot\,\xi)$.
\end{theorem}

\medskip


Let us consider a set of $ p$  weak almost contact structures with the same tensor $Q$ on an $(n+p+np)$-dimensional manifold $M$:
$(M, \xi_i, \eta_i, \phi_i, Q),\ i=1,\ldots,p$.
 If this set satisfies the following conditions:
\begin{eqnarray}
\label{E-Q0-p}
 && \phi_i\circ\phi_j = -\delta_{ij}Q + \,\eta^j\otimes\xi_i
 +\sum\nolimits_k \eps_{ijk}\,\phi_k,\quad i,\,j,\,k\in\{1,\ldots, p\},\\
 \label{E-Q0-pa}
 && Q\,\xi_i=\xi_i,\quad i\in\{1,\ldots, p\},
\end{eqnarray}
where $\eps_{ijk}$ is the completely antisymmetric symbol, i.e.,
changes sign under exchange of each pair of its indices, then it will be called a \textit{weak almost $p$-contact structure}.
If~there exists a metric $g$ \textit{compatible} with each of our $p$ weak almost contact structures,
\begin{equation}\label{E-Q2-p}
 g(\phi_i X,\phi_i Y)=g(X,QY)-\eta^i(X)\eta^i(Y),
\end{equation}
then we say that it is  a \textit{weak almost $p$-contact metric structure} on the manifold $M.$

For $Q=\id$, we get the structure considered in \cite{blaga},
and for $p=3$, it generalizes an {almost 3-contact (metric) structure} on $M^{4n+3}$.
From \eqref{E-Q0-pa} it follows that (see also \cite{blaga} for $Q=\id$)
\[
 \phi_i\,\xi_j=\sum\nolimits_k \eps_{ijk}\,\xi_k,\quad
 \eta^i\circ \phi_j =\sum\nolimits_k \eps_{ijk}\,\phi_k.
\]
By~\eqref{E-Q0-pa} and \eqref{E-Q0-p} the tensor $Q$ is nonsingular, and by \eqref{E-Q2-p} $Q$ is self-adjoint.
Observe that for a weak almost $p$-contact metric structure, the tensors $\phi_i$ are skew-symmetric,
\[
 g(\phi_i X, Y) = -g(X, \phi_i Y),\quad i\in\{1,\ldots p\},
\]
and the Reeb vector fields $\{\xi_i\}$ are orthonormal with respect to~$g$.
 Define complementary ortho\-gonal distributions $\widetilde{\cal D}={\rm span}(\xi_1,\ldots,\xi_p)$
and ${\cal D}=\bigcap_{\,i}\ker\eta_i$. Then $Q\,|_{\widetilde{\cal D}}=\id^\top$, see \eqref{E-Q0-pa}.

For this structure the following version of Proposition~\ref{T-04} is valid.

\begin{proposition}
If the tensor fields $\phi_i$ of a weak almost $p$-contact structure
have a skew-symmetric representation with the same local frames, then this structure admits a compatible metric.
\end{proposition}

\begin{definition}\rm A weak almost $p$-contact metric structure  $(\phi_i,\xi_i,\eta^i,Q),\ i=1,\ldots,p$, on a manifold $(M,g)$
is called a \textit{weak $p$-Sasakian structure} if each of the structures $(\phi_i,\xi_i,\eta^i,Q)$
is a {weak Sasakian structure} on $(M,g)$, that is
\begin{equation}\label{E-wq-p-S}
 g(X,\phi_i Y)=d\eta^i(X,Y),
\end{equation}
\end{definition}

The \textit{second fundamental tensor} ${h}:{\cal D}\times{\cal D}\to\widetilde{\cal D}$
and the \textit{integrability tensor} ${T}:{\cal D}\times{\cal D}\to\widetilde{\cal D}$ of the distribution ${\cal D}$ are given by
\begin{equation}\label{E-def-bT}
 h(X,Y)=(1/2)(\nabla_X Y+\nabla_Y X)^\top,\quad
 T(X,Y) =(1/2)(\nabla_X Y-\nabla_Y X)^\top.
\end{equation}
The \textit{shape operator} ${A}_\xi$
and the skew-symmetric operator ${T}^{\sharp}_\xi$ are given, respectively, by:
\begin{equation*}
 g({A}_\xi X,Y)= g({h}(X,Y),\xi),\quad
 g({T}^{\sharp}_\xi X,Y)=g({T}(X,Y),\xi),\quad
 X,Y \in {\cal D},\ \xi \in \widetilde{\cal D}.
\end{equation*}
Since $T(X,Y)=\frac12\,[X, Y]^\top$,
the distribution ${\cal D}$ is tangent to a foliation if and only if $T=0$.
If~$h=0$ then ${\cal D}$ is a \textit{totally geodesic} distribution.
For a totally geodesic foliation $\calf$ the equalities $T=h=0$ are satisfied.

The condition \eqref{E-wq-p-S} of a weak $p$-Sasakian structure ensures that
\begin{equation}\label{E-phi-T}
 T_{\xi_i}^\sharp=\phi_i\,_{|\,\cal D}.
\end{equation}
Now, let us recall the definition of the partial Ricci curvature tensor. Let $(M,g)$ be a connected
Riemannian manifold, $\widetilde{\cal D}=T\calf$ the tangent distribution to an $n$-dimensional foliation $\calf$
and ${\cal D}$ the normal (i.e., orthogonal to $\widetilde{\cal D}$) distribution of dimension~$p$.
Denote $^\top$ and $^\perp$ orthogonal projections onto $\widetilde{\cal D}$ and ${\cal D}$, respectively.
A~local adapted orthonormal frame $\{E_i,\,{\cal E}_{j}\}$, where $\{E_i\}\subset\widetilde{\cal D}$ and $\{{\cal E}_j\}\subset{\cal D}$,
always exists on~$M$. The \textit{partial Ricci curvature tensor}
\begin{equation*}
 r_g(X,Y)=\sum\nolimits_{i} {R}(X^\bot,\,E_i, Y^\bot, E_i),\quad X,Y\in {\cal X}_M,
\end{equation*}
is the symmetric (0,2)-tensor on $(M,g,\widetilde{\cal D})$, see \cite{r2010}, and
its adjoint (1,1)-tensor~is
\[
 \Ric^\bot(X)=\sum\nolimits_{i} ({R}(E_i\,,X^\bot)E_i)^\bot.
\]

\begin{remark}\rm
The partial Ricci curvature is an important invariant of almost product manifolds and foliated Riemannian manifolds.
The trace of this rank 2 tensor is the mixed scalar curvature $S_{\rm mix}$, defined as an averaged sectional curvature of planes that non-trivially intersect $\widetilde{\cal D}$ and ${\cal D}$, and examined by several geometers,
recall just integral formulas and splitting results, curvature prescribing and variational problems for foliations, see survey \cite{rov-5}.
The understanding of $\Ric^\bot$ and $S_{\rm mix}$ is a fundamental problem of extrinsic geometry of foliations, see \cite{rov-2}.
\end{remark}

\begin{remark}
\rm
There are several possible options when the dimension of the characteristic foliation is of dimension greater than 1.
In the case of dimension 3, an (almost) contact metric structure can be replaced by an (almost) contact metric $3$-structure, defined as a set of 3 (almost) contact structures, $(M,\phi_i,\xi_i,\eta^i)$,
\[
 \phi_i^2=-\id +\,\eta^i\otimes\xi_i,\quad \eta^i(\xi_j)=\delta^i_j
\]
with the same compatible metric $g$, i.e., $g(\phi_i X,\phi_i Y)=g(X,Y)-\eta^i(X)\,\eta^i(Y)$, obeying
\[
 \phi_k = \phi_i \circ \phi_j - \eta^i \otimes \xi_j = - \phi_j \circ \phi_i + \eta^j \otimes \xi_i
\]
for any cyclic permutation $(i,j,k)$ of $(1,2,3)$, see \cite{b2010}.
The~dimension of $M$ with an almost contact 3-structure is $4n+3$.
We get a 3-\emph{Sasakian structure} if each of $(\phi_i,\xi_i,\eta^i)$ is Sasakian.
For a 3-Sasakian structure, $[\xi_i,\xi_j]=c\,\xi_k$ holds for some $c\in\RR$
and any cyclic permutation $(i,j,k)$ of $(1,2,3)$; thus,
$\widetilde{\cal D}={\rm span}(\xi_1,\xi_2,\xi_3)$ is integrable; moreover, it defines a totally geodesic Riemannian foliation
with the property \eqref{E-phi-T}.
Hence, $r_g(X,X)=3$ for $X\bot\,{\cal D}$ and $|X|=1$.
\end{remark}

Finally, we are able to formulate the main result, whose proof is based on some technical results, which are formulated and demonstrated in  Section~\ref{sec:flow}.

Define a $(0,2)$-tensor $g^\bot$ by $g^\bot(X,Y):=g(X^\bot,Y^\bot)$ for
$X,Y\in {\cal X}_M$.

\begin{theorem}\label{T-PRF-g2}
Let $(\phi_i,\xi_i,\eta^i,Q)$ be a weak $p$-Sasakian structure on a Riemannian manifold $(M,g_0)$
such that $\widetilde{\cal D}={\rm span}(\xi_1,\ldots,\xi_p)$ determines a $\mathfrak{g}$-foliation.
If~$\,r_{g_0}>0$ on the orthogonal distribution ${\cal D}$, then
there exists a smooth family of metrics $g(t)\ (t\in\RR)$
such that $(\phi_i(t),\xi_i,\eta^i,Q_t)$ is a weak $p$-Sasakian structure on $(M,g_t)$ with
$Q$ and $\phi$ redefined on ${\cal D}$ as
\begin{equation}\label{E-Q-phi}
 Q_t=(1/p)\Ric^\bot_t,\quad \phi_i(t)\,|_{\,{\cal D}} = T_{\xi_i}^\sharp(t).
\end{equation}
Moreover, $g_t$ converges exponentially fast, as $t\to-\infty$, to a limit metric $\hat g$
with $r_{\hat g}=p\,\hat g^\bot$.
\end{theorem}

\begin{proof}
Consider the partial Ricci flow, see the formula \eqref{E-GF-Rmix-Phi} in Section~\ref{sec:flow},
of metrics $g_t$ on $M$ with constant $\Phi=p$. 
In~our case, for $t=0$, by the formula \eqref{E-genRicN} of Section~\ref{sec:flow}
and \eqref{E-phi-T}, we get on ${\cal D}$:
\[
 \Ric^\bot = -\sum\nolimits_{\,i} (T_{\xi_i}^\sharp)^2 = -\sum\nolimits_{\,i} \phi_i^2 = p\,Q ,
\]
in particular, $Q>0$, and using \eqref{E-Q0-p} and \eqref{E-phi-T}, we find $T_{\xi_i}^\sharp \Ric^\bot =\Ric^\bot T_{\xi_i}^\sharp$.
Thus,
\[
 \sum\nolimits_i T_{\xi_i}^\sharp \Ric^\bot T_{\xi_i}^\sharp = -(\Ric^\bot)^2.
\]
By~the above, Lemma~\ref{P-PRF-g1} and the formula \eqref{E-RbotT-g2}
of Section~\ref{sec:flow}, we obtain the following ODE:
\[
 \dt\Ric^\bot = 4\Ric^\bot(\Ric^\bot -\,p\,\id^\bot).
\]
By the theory of ODE's, there exists a unique solution $\Ric^\bot(t)$ for $t\in\RR$;
hence, a solution $g_t$ of \eqref{E-GF-Rmix-Phi} exists for $t\in\RR$ and is unique.
Observe that $(\phi_i(t),\xi_i,\eta^i,Q_t)$ with $\phi_i(t),Q_t$ given in \eqref{E-Q-phi}
is a weak $p$-Sasakian structure on $(M,g_t)$.
By uniqueness of a solution, the partial Ricci flow preserves the directions of eigenvectors of $\Ric^\bot$ and
each eigenvalue $\mu_i$ of $\Ric^\bot$ satisfies ODE (for $\mu=\mu_i$)
\[
 \dot\mu_i=4\mu_i\,(\mu_i-p).
\]
This ODE has the following solution (function $\mu(t)$ on $M$ for any $t\in\RR$):
\[
 \mu_i(t)= \frac{\mu_i(0)\,p}{\mu_i(0)+\exp(4\,p\,t)(p-\mu_i(0))}
\]
with $\mu_i(0)>0$ and $\lim\limits_{t\to-\infty}\mu_i(t)=p$.
Thus, $\lim\limits_{t\to-\infty}\Ric^\bot(t)=p\,\id^\bot$.
Let $\{e_i(t)\}$ be a $g_t$-orthonormal frame of ${\cal D}$ of eigenvectors associated with $\mu_i(t)$.
We then have
 $\dt e_i = (\mu_i - p) e_i$.
Since $e_i(t)=z_i(t)\,e_i(0)$ with $z_i(0)=1$, then $\dt \log z_i(t) = \mu_i(t) - p$.
By the~above, $z_i(t) = (\mu_i(t)/\mu_i(0))^{1/4}$. Hence,
\[
 g_t(e_i(0),e_j(0))= z_i^{-1}(t)z_j^{-1}(t)\,g_t(e_i(t),e_j(t)) = \delta_{ij}(\mu_i(0)\mu_j(0)/(\mu_i(t)\mu_j(t)))^{1/4}.
\]
As $t\to-\infty$, $g_t$ converges to the metric $\hat g$ determined by
$\hat g(e_i(0),e_j(0))=\delta_{ij}\sqrt{\mu_i(0)/p}$.
\end{proof}

\begin{corollary}
Let $(\phi_i,\xi_i,\eta^i,Q)$ be a weak 3-Sasakian structure on $(M^{4n+3},g_0)$
such that $\widetilde{\cal D}={\rm span}(\xi_1,\xi_2,\xi_3)$ determines a $\mathfrak{g}$-foliation.
If~$\,r_{g_0}>0$ on the orthogonal distribution ${\cal D}$ then there exists a smooth family of metrics $g_t\ (t\in\RR)$
such that $(\phi_i(t),\xi_i,\eta^i,Q_t)$ is a weak $3$-Sasakian structure on $(M,g_t)$ with
$Q$ and $\phi$ redefined on ${\cal D}$ as
\[
 Q_t = (1/3)\Ric^\bot_t,\quad \phi_i(t)\,|_{\,{\cal D}} = T_{\xi_i}^\sharp(t).
\]
Moreover, $g_t$ converges exponentially fast, as $t\to-\infty$, to a limit metric $\hat g$ with $r_{\hat g}=3\,\hat g^\bot$,
that gives a {$3$-Sasakian structure}.
\end{corollary}

\begin{proof}
By Theorem~\ref{T-PRF-g2} with $p=3$, the metric $\hat g$ determines a 3-Sasakian structure.
\end{proof}

A nonsingular Killing vector clearly defines a Riemannian flow; moreover, a Killing vector of unit length
generates a geodesic Riemannian flow.
Recall \cite{b2010} that a $K$-\textit{contact structure} is a contact metric structure, for which the characteristic (Reeb) vector field is Killing.

The following corollary of Theorem~\ref{T-PRF-g2} for $p=1$
generalizes \cite[Proposition~7.4]{b2010} where $R_\xi=\id|_{\,\cal D}$.

\begin{corollary}
Let a Riemannian manifold $(M^{2n+1},g_0)$ admit a unit Killing vector field $\xi$ such that the Jacobi operator $R_\xi$ is positive definite on the  distribution orthogonal to $\xi$. Then there exists a smooth family of metrics $g_t\ (t\in\RR)$,
which converges exponentially fast, as $t\to-\infty$, to a limit metric $\hat g$ with $r_{\hat g}=\hat g^\bot$
that gives a $K$-contact structure on $M$.
\end{corollary}

\section{Weak $f$-structures}
 Classical $f$-structures can be considered to be higher dimensional analogs of almost contact structures. In this section we propose the study of their "weak" version.

\begin{definition}\label{D-weak_f}\rm
A \textit{weak $f$-structure} on a manifold $M^{2n+p}$ is defined by a $(1,1)$-tensor field $f$ of rank $2n$
and a~nonsingular $(1,1)$-tensor field $Q$ satisfying
\begin{eqnarray}\label{E-fQ-1}
\nonumber
 && f^3 +fQ = 0,\\
 && Q\,\xi=\xi,\quad \xi\in\ker f.
\end{eqnarray}
If there exist vector fields $\xi_i,\,i\in\{1,\ldots, p\}$,
with their dual 1-forms $\eta^i$, satisfying
\begin{equation}\label{E-fQ-2}
 f^2 = -Q +\sum\nolimits_i\eta^i\otimes\xi_i,\quad \eta^i(\xi_j)=\delta^i_j.
\end{equation}
then we obtain a \textit{weak globally framed $f$-structure}.
\end{definition}

A weak globally framed $f$-structure is said to be \textit{normal} if
$N_f + 2\sum\nolimits_{\,i} d\eta^i\otimes\xi_i =0$.

Notice that $TM$ splits into two complementary subbundles ${\cal D}=f(TM)$ and $\widetilde{\cal D}=\ker f$.
Hence $Q$ satisfies
$Q\,|_{\,\widetilde{\cal D}}=\id|_{\,\widetilde{\cal D}}$.

\medskip

Similarly to Example~\ref{Ex-06}, we can construct an example of a {weak almost $f$-manifold}.

It is easy to show that for a weak globally framed $f$-{structure} $(f,\xi_i,\eta^i,Q)$,
\[
 f\,\xi_i=0,\quad \eta^i\circ f=0,\quad i\in\{1,\ldots,p\}.
\]

\noindent
 By \eqref{E-fQ-2}, $f^2\xi_i=0$.
Applying \eqref{E-fQ-1} to $\xi_i$, we get $f\,\xi_i=0$.
 To show $\eta^i\circ f=0$, observe that from \eqref{E-fQ-2} and $Q\,\xi_i=\xi_i$ we get
 $[Q,\,f]=0$.
Since $f\,\xi_i=0$, we also have, applying \eqref{E-fQ-2},
\[
 \eta^i( f X)= f(f^2 X) +Q(f X) = Q(f X) -f(Q X) = [Q, f](X) = 0
\]
for any $X\in {\cal X}_M$, that proves the claim.

A Riemannian metric $g$ is \textit{compatible} with a weak globally framed $f$-structure if
\begin{equation}\label{E-compatible-g-f2}
 g(f X, f Y) = g(X,QY) - \sum\nolimits_i \eta^i(X)\,\eta^i(Y).
\end{equation}
A weak globally framed $f$-structure with a compatible Riemannian metric is called \textit{ a metric weak $f$-structure.}
From \eqref{E-compatible-g-f2} we get $g(X, \xi_i) =\eta^i(X)$,
$\{\xi_i\}$ are orthonormal with respect to~$g$, the tensor $Q$ is self-adjoint and the tensor $f$ is skew-symmetric,
\[
 g(Q X, Y) = g(X, Q Y),\quad
 g(f X, Y) = -g(X, f Y).
\]

\noindent
Similarly to weak almost contact structures, any weak globally framed $f$-structu\-re admits a compatible Riemannian metric if we assume one additional obvious condition, namely,

\begin{proposition}
If the tensor field $f$ of a weak globally framed $f$-structure has a skew-symmetric representation
then the structure admits a compatible metric.
\end{proposition}

\begin{proof}
This is similar to the proof of Proposition~\ref{T-04}.
\end{proof}

Finally, we define the Sasaki 2-form $F$ putting $F(X, Y) = g(X, f(Y))$ for $X, Y \in {\cal X}_M$.
A metric weak globally framed $f$-manifold will be called a \textit{weak ${\cal K}$-manifold} if it is normal and $d F = 0$.
Two subclasses of weak ${\cal K}$-manifolds can be defined as follows:
\textit{weak almost ${\cal C}$-manifolds} if $d\eta^i = 0$ for any $i$,
and \textit{weak almost ${\cal S}$-manifolds} if $d\eta^i = F$ for any $i$.

\begin{remark}\rm
For $Q=\id$, where $\id : TM \to TM$ is the identity mapping, Definition~\ref{D-weak_f} gives
an $f$-structure, see \cite{Na1966}, and a globally framed $f$-structure, see \cite{gy}.
For an $f$-{manifold}, $TM$ splits into sum of subbundles ${\cal D}=f(TM)$ and $\widetilde{\cal D}=\ker f$,
and that the restriction of $f$ to ${\cal D}$ determines a complex structure on it.
 On~a ${\cal K}$-manifold there exists a $p$-parameter group of isometries generated by the set of Killing vector fields $\{\xi_i\}$, see \cite[Theorem~1.1]{b1970}.
Since the sectional curvature is $K(X,\xi_i)=1$ for unit $X$, then
\[
 r_g(X,X)=p \quad (X\in{\cal D},\ |X|=1).
\]
\end{remark}

We complete this section with the formulation of our main result on weak $f$-structures.

\begin{theorem}\label{T-PRF-g2b}
Let $(f,\xi_i,\eta^i,Q),\ i=1,\ldots,p,$ be a metric weak almost ${\cal S}$-structure on a Riemannian mani\-fold $(M,g_0)$ and $\widetilde{\cal D}={\rm span}(\xi_1,\ldots,\xi_p)$ be the tangent distribution  to a $\mathfrak g$-foliation.
If~$\,r_{g_0}>0$ on the orthogonal distribution ${\cal D}$, then there exists a smooth family of metrics $g_t\ (t\in\RR)$
such that $(f(t),\xi_i,\eta^i,Q_t)$ is a weak almost ${\cal S}$-structure on $(M,g_t)$ with $Q$ and $f$ redefined on ${\cal D}$ as
\[
 Q_t=(1/p)\Ric^\bot_t,\quad f(t)\,|_{{\cal D}}=T_{\xi_i}^\sharp(t).
\]
Moreover, $g_t$  converges exponentially fast, as $t\to-\infty$, to a limit metric $\hat g$ with $r_{\hat g}=p\,\hat g^\bot$,
that gives a {metric almost ${\cal S}$-structure}.
\end{theorem}

\begin{proof} The proof is analogous to the proof of Theorem~\ref{T-PRF-g2},
therefore we leave it to the reader as an exercise.
\end{proof}

\section{Weak para-$\phi$-structure}
\label{sec:flow4}

Classical almost para-$\phi$-structures generalizes the almost product and the almost paracontact structures.
In this section we propose the study of its "weak" version.

\begin{definition}\label{D-weak_para-phi}\rm
A \textit{weak para-$\phi$-structure} on a manifold $M^{2n+p}$ is defined by a $(1,1)$-tensor field $\phi$
of rank $2n$ and a nonsingular $(1,1)$-tensor field $Q$ satisfying
\begin{eqnarray*}
 && \phi^3-\phi\,Q=0,\\
 && Q\,\xi=\xi,\quad \xi\in\ker\phi.
\end{eqnarray*}
If there exist vector fields $\xi_i,\ 1\le i\le p$,
and dual 1-forms $\{\eta^i\}$, satisfying the following conditions:
\begin{equation*}
 \phi^2 = Q - \sum\nolimits_i\eta^i\otimes \xi_i,\quad \eta^i(\xi_j)=\delta^i_{j},
\end{equation*}
then $M$ is called a \textit{weak almost para-$\phi$-manifold with complemented frames}, or, in short, a weak almost para-$\phi$-manifold.
\end{definition}

The kernel distribution $\widetilde{\cal D}=\ker\phi$ has dimension~$p$. Set ${\cal D}=\phi(TM)$.
Hence, $Q|_{\,\widetilde{\cal D}}=\id_{\widetilde{\cal D}}$.

In a similar way as in  Example~\ref{Ex-06}, we can construct an example of a weak almost para-$\phi$-manifold.
Using standard calculations, we can show that if the manifold $M^{2n+p}$ has a \textit{weak para-$\phi$-structure} $(\phi,\xi_i,\eta^i,Q)$, then
\[
 \phi\,\xi_i=0,\quad  \eta^i\circ\phi=0,\quad i\in\{1,\ldots,p\}.
\]
A weak almost para-$\phi$-manifold $M$ with a \textit{compatible} Riemannian metric $g$, that is
\begin{equation*}
 g(\phi X, \phi Y) = -g(X,QY) + \sum\nolimits_i \eta^i(X)\eta^i(Y),
\end{equation*}
is called a \textit{metric weak almost para-$\phi$-manifold}.
By \eqref{E-Q2-g}, a metric weak almost para-$\phi$-manifold has self-adjoint $Q$ and skew-symmetric $\phi$,
\[
 g(Q X, Y) = g(X, Q Y),\quad
 g(\phi X, Y) = -g(X, \phi Y),
\]
and the vector fields $\{\xi_i\}$ are orthonormal with respect to~$g$.

It is not difficult to demonstrate as for other structures that if $\phi$ of a weak para-$\phi$-structure
has a skew-symmetric representation, then the structure admits a compatible metric.

A weak almost para-$\phi$-structure is said to be \textit{normal}~if $N_\phi - 2\sum\nolimits_i d\eta^i\otimes\xi_i=0$.

 On a metric weak almost para-$\phi$-manifold, we define a 2-form by
\[
 F(X, Y) = g(X, \phi(Y)),\quad X, Y \in {\cal X}_M.
\]
A \textit{weak para-$S$-manifold} is a normal weak almost para-$\phi$-manifold with $F=d\eta^i$ for any $i$.

\begin{remark}\rm
For $Q=\id$, where $\id : TM \to TM$ is the identity mapping, Definition~\ref{D-weak_para-phi} gives
a para-$\phi$-structure and a weak almost para-$\phi$-manifold, see \cite{fp-m,tar}.
The~kernel distribution $\widetilde{\cal D}=\ker\phi$ has dimension~$p$ and $\pm1$ eigen-distributions of $\phi$,
denoted by ${\cal D}^+$ and ${\cal D}^-$, respectively, have the same dimension equal to $n$.
Set ${\cal D}:=\phi(TM)={\cal D}^+\oplus {\cal D}^-$.
For a {para-$S$-manifold}, $\nabla_X\,\xi_i = -\phi(X)$ for all $i$, see \cite{fp-m,od}, and
\[
 R(X,Y)\xi_i = \sum\nolimits_j \big[\eta^j(X)\phi^2(Y) -\eta^j(Y)\phi^2(X)\big].
\]
For $X=\xi_i$, $Y\in{\cal D}$ we get $R(\xi_i,Y)\xi_i = Y$. Hence,
 $r_g(Y,Y)=p$ for all $Y\in{\cal D},\ |Y|=1$.
\end{remark}

\begin{theorem}\label{T-PRF-g2c}
Let $(\phi,\xi_i,\eta^i,Q)$ be a metric weak para-${\cal S}$-structure on $(M,g)$
and the distribution $\widetilde{\cal D}={\rm span}(\xi_1,\ldots,\xi_p)$  be tangent to a $\mathfrak{g}$-foliation.
If~$\,r_{g_0}>0$ on the orthogonal distribution ${\cal D}$ then there exists a smooth family of metrics $g_t\ (t\in\RR)$
such that $(\phi(t),\xi_i,\eta^i,Q_t)$ is a metric weak para-${\cal S}$-structure on $(M,g_t)$ with $Q$ and $\phi$ redefined on ${\cal D}$ as
\[
 Q_t=(1/p)\Ric^\bot_t,\quad \phi(t)\,|_{\,{\cal D}}=T_{\xi_i}^\sharp(t).
\]
Moreover, $g_t$ converges exponentially fast, as $t\to-\infty$, to a limit metric $\hat g$
with $r_{\hat g}=p\,\hat g^\bot$, that gives a {metric almost para-${\cal S}$-structure}.
\end{theorem}

\begin{proof}
 The proof  is  analogous to the proof of Theorem~\ref{T-PRF-g2}, therefore we omit it.
\end{proof}

\section{Technical results}
\label{sec:flow}

In the study of flows there are two important problems to consider: the limit sets and the stationary/fixed points.
Here, we will investigate the first problem (the second problem is considered in Sections~\ref{sec:main}--\ref{sec:flow4}).
Recall that a flow of Riemannian metrics $g_t$ on a smooth manifold $M$ is an evolution of a geometric structure, e.g.,~\cite{ah},
\begin{equation}\label{E-egf0}
 \dt g = B(g),
\end{equation}
where $B(g)$ is a symmetric (0,2)-tensor field.
The Ricci flow appears when $B(g)=-2\Ric(g)$.
Here, $\Ric=\tr_{\,24}R$ is the Ricci curvature of the curvature tensor $R$.

The Levi-Civita connection for \eqref{E-egf0} evolves as, e.g.,~\cite{ah},
\begin{equation}\label{eq2}
 2\,g_t((\dt\nabla^t)(X, Y), Z)=(\nabla^t_X\,B_t)(Y,Z)+(\nabla^t_Y\,B_t)(X,Z)-(\nabla^t_Z\,B_t)(X,Y)
\end{equation}
for all $X,Y,Z\in {\cal X}_M$.
Next, we recall some notions and results of \cite{r24}.

\begin{proposition}[see \cite{r24}]
The geometric quantities of a totally geodesic foliation related to ${\cal D}$ evolve by \eqref{E-egf0}
$($with a symmetric $(0,2)$-tensor $B(g)\,)$ according to
\begin{eqnarray}
\label{E1-S-A}
 2\,\dt A_\xi\eq -\nabla_{\xi}\,B^\sharp +[ A_\xi- T^\sharp_\xi,\ B^\sharp],\qquad
 \dt T^\sharp_\xi = -B^\sharp\,T^\sharp_\xi,\qquad \xi\in\widetilde{\cal D}.
\end{eqnarray}
\end{proposition}

\begin{proof}
Note that $\dt T=0$.
For all $X,Y\in{\cal D}$, using (\ref{eq2}) and (\ref{E-def-bT}), we find
\begin{eqnarray*}
 2\,g\big(\dt(\nabla^t_X Y),\,\xi\big)
 \eq (\nabla^t_X B)(Y,\xi)+(\nabla^t_Y B)(X,\xi) -(\nabla^t_\xi B)(X,Y)\\
 \eq -(\nabla^t_\xi B)(X,Y) -B(Y,\nabla^t_X \xi) -B(X,\nabla^t_Y\,\xi).
\end{eqnarray*}
This and symmetry of $\dt\nabla^t$ give
\begin{equation}\label{E1-S-b}
 2\,g(\dt h(X,Y),\xi) = -(\nabla_\xi\,B)(X,Y) +B(Y, C_\xi\,X) +B(X, C_\xi\,Y).
\end{equation}
Here, the \textit{co-nullity (splitting) tensor} is defined by
\begin{equation*}
 C_\xi\,X=-(\nabla_{X}\, \xi)^\bot,\quad
 X \in {\cal D},\ \ \xi\in\widetilde{\cal D}.
\end{equation*}
\noindent
We~have the identities
\begin{equation}\label{E-CC**}
 A_\xi=(C_\xi + C_\xi^{\,*})/2,\quad T^\sharp_\xi =(C_\xi - C_\xi^{\,*})/2,\quad
 C_\xi = A_\xi + T^\sharp_\xi.
\end{equation}
Using
\begin{equation*}
 g(A_\xi(X),\,Y) = g(h(X,Y),\,\xi),\quad g(T^\sharp_\xi(X),\,Y) = g(T(X,Y),\,\xi)\quad (X,Y\in{\cal D}),
\end{equation*}
we then find
\begin{eqnarray*}
 g(\dt A_\xi(X),Y) = \dt g(h(X,Y),\xi) -(\dt g)(A_\xi(X),Y),\\
 g(\dt T^\sharp_\xi(X),Y) = -(\dt g)(T^\sharp_\xi(X),Y).
\end{eqnarray*}
The above, \eqref{E1-S-b} and \eqref{E-CC**}, yield (\ref{E1-S-A}).
\end{proof}

Some authors consider flows of metrics on a foliated manifold with the metric varying along transverse (to the leaves) distribution,
e.g., second-order quasilinear transversally parabolic flows \cite{bhv}, which
can be applied to other flows like the transverse Ricci flow and Sasaki-Ricci flow. In~\cite{r4}--\cite{RWo-1},
they study flows of metrics on a foliation called \textit{extrinsic geometric flows}: although the metric varies along normal to the leaves distribution,
such flows are parabolic along the~leaves.

An analogue in a sense of the Ricci flow for foliations, is the \textit{partial Ricci flow}, see \cite{r24}.

\begin{definition}[see~\cite{r24}]\rm
The \textit{normalized~partial Ricci flow} is defined~by
\begin{equation}\label{E-GF-Rmix-Phi}
 \dt g = -2\,r_g +2\,\Phi\,g^\bot,
\end{equation}
where $g=g^\top+ g^\bot$
and $\Phi:M\to{\mathbb R}$ is a leaf-wise constant function.
\end{definition}

The flow \eqref{E-GF-Rmix-Phi}, preserves metric on the leaves of $\calf$ and the orthogonality of vectors to $T\calf$; if~$\calf$ is either totally umbilical, totally geodesic or harmonic foliation for $t=0$ then it has the same property for all $t>0$.
It was proposed \cite{r24} as the main tool to prescribe the partial Ricci and the mixed sectional curvature of a totally geodesic foliation.

The principal difference of the partial Ricci flow from other known flows with metric varying along ${\cal D}$ is that the PDE's under consideration are parabolic along the leaves and not along ${\cal D}$.

\smallskip

Next, we study tangentially Lie foliations, whose characteristic foliation is regular and its dimension
is equal to the dimension of the Lie algebra $\mathfrak{g}$. We restrict our attention to \textit{compatible metrics} that
is bundle-like metrics, for which the foliation is totally geodesic and the chosen characteristic vector fields $ \xi_i$ are orthonormal.
For compatible metrics, the characteristic foliation is Riemannian, and the mixed sectional curvature is nonnegative; thus $\Ric^\bot\ge0$.

\begin{remark}\label{R-01}\rm
Consider the behavior of tensor fields associated to a Riemannian metric with respect to diffeomorphisms.
Let $(M,g)$ be a Riemannian metric and let $f:M \rightarrow M'$ be a diffeomorphism. Then $f$ is an isometry between $(M,g)$ and $(M',g'=f_{*}g)$. The Levi-Civita connections $\nabla^g$ and $\nabla^{g'}$ are $f$-related, i.e.,
\[
 df(\nabla_X^gY) = \nabla_{df(X)}^{g'}df(Y)
\]
for any vector fields $X,Y$ on $M.$ Therefore, cf. \cite[Propositions~VI.1.2 and VI.1.4]{KN}, the torsion tensor $T^g$ is $f$-related to $T^{g'}$  and the curvature tensor field $R^g$ is $f$-related to $R^{g'}$. If we just consider $f$-related orthonormal bases then any tensor field $B$ obtained via contractions, traces etc., are $f$-related, i.e.,
\[
 df(B(g))_x  = B(g')_{f(x)}.
\]
In the case of ``partial" tensor fields, we have to consider diffeomorphisms $f$, which preserve the foliation, i.e.,
 $df(T{\mathcal F}) \subset T{\mathcal F}$.
Then $ df (T{\mathcal F}^{\perp _g} )  \subset T{\mathcal F}^{\perp _{g'}}$; thus, ``partial" tensors are also related.
\end{remark}

\begin{lemma}\label{P-PRF-g1}
 The class of compatible metrics of $\mathfrak{g}$-foliations is preserved by the flow \eqref{E-GF-Rmix-Phi}.
\end{lemma}

\begin{proof} Consider solution of \eqref{E-egf0} with the initial condition $g=g_0$ -- a compatible metric,
where $B(g)=-2\,r_g$ is a symmetric (0,2)-tensor depending on the metric $g$.
 If a~smooth diffeomorphism $\phi \colon M \rightarrow M$ is an isometry of $g$, then $\phi^*g=g$.
Therefore, $\phi^*r_g = r_{\phi^*g} = r_g$ for any isometry $\phi$ of $g$ preserving $\widetilde{\cal D}$,
see Remark~\ref{R-01}. Then
\[
 \partial_t (\phi^*g_t ) = \phi^*(\partial_t g_t ) = \phi^*(r_{g_t}) = r_{\phi^*g_t}.
\]
Hence the family $\phi^*g_t $ is the evolution of $\phi^*g_0 =g_0$.
From the uniqueness of solution (the linearization of \eqref{E-GF-Rmix-Phi} at $g_0$
is a leaf-wise parabolic PDE, see \cite{r24}), we get $\phi^*g_t = g_t $ for any $t$.
Thus,~the isometry $\phi$ of $g$ is also an isometry of any metric of its evolution.
\end{proof}

\begin{lemma}
The following equalities hold for any $\mathfrak{g}$-foliation:
\begin{eqnarray}\label{E-genricA-W}
 R(\xi_i, X)\xi_j \eq\nabla_{\xi_i}T^\sharp_{\xi_j}(X)-T^\sharp_{\xi_j} T^\sharp_{\xi_i}(X),\quad
 X\in{\cal D},\ \xi_i,\xi_j \in \widetilde{\cal D},\\
\label{E-genRicN-W}
 \Ric^\bot \eq -\sum\nolimits_{i} (T_{\xi_i}^\sharp)^2,
\end{eqnarray}
 $\|T\|$ is leafwise constant and
\begin{equation}\label{E-AT-g}
 \nabla_{\xi}\Ric^\bot = [ \Ric^\bot,\ T^\sharp_\xi ], \quad \xi \in \widetilde{\cal D}.
\end{equation}
\end{lemma}

\begin{proof}
The following equalities hold for a totally geodesic foliation ${\cal F}$, see \cite{r24}:
\begin{eqnarray}\label{E-genricA}
 R(\xi_1, X)\xi_2 \eq (\nabla_{\xi_1} C)(\xi_2,X) - C(\xi_2, C(\xi_1, X)),\quad
 \xi_1,\xi_2\in\widetilde{\cal D},\ X\in{\cal D},  \\
\label{E-genRicN}
 \Ric^\bot \eq \Div_{\calf} h -\sum\nolimits_{i}\big(A_{\xi_i}^2+(T_{\xi_i}^\sharp)^2\big)^\flat.
\end{eqnarray}
where the $\calf$-divergence of a $(1,2)$-tensor $h$ is a $(0,2)$-tensor
\[
 (\Div_\calf h)(X_1,X_{2})
 =\sum\nolimits_{i} g((\nabla_{\xi_i} h) (X_1,X_2), \xi_i).
\]
By Lemma~\ref{P-PRF-g1}, assumption $h=0,$ and \eqref{E-genricA} and \eqref{E-genRicN}, we obtain
\eqref{E-genricA-W} and \eqref{E-genRicN-W}.
By~Lemma~\ref{P-PRF-g1}, one may consider the partial Ricci flow family $g_t$ of compatible metrics.
By \eqref{E1-S-A}$_1$ with $B=-2\,r_g+2\Phi\,g^\bot$, using $A_\xi=0$ and $\nabla_\xi\,\id^\bot =0$ (for Riemannian foliations), we get \eqref{E-AT-g}.
Taking trace of \eqref{E-AT-g} and using \eqref{E-genRicN-W}, we find $\xi(\|T\|^2)=0$.
\end{proof}

\begin{proposition}
The flow \eqref{E-GF-Rmix-Phi} for compatible metrics on $\mathfrak{g}$-foliations obeys the ODE's
\begin{eqnarray}\label{E1-b-T-Riem}
 \dt T^\sharp_\xi \eq 2(\Ric^\bot - \Phi\id^\bot)\,T^\sharp_\xi,\\
\label{E-RbotT-g2}
 \dt\Ric^\bot\eq 2\Ric^\bot(\Ric^\bot -2\Phi\id^\bot) -2\sum\nolimits_i T_{\xi_i}^\sharp \Ric^\bot T_{\xi_i}^\sharp.
\end{eqnarray}
\end{proposition}

\begin{proof}
By \eqref{E1-S-A}$_2$ with $B=-2\,r_g+2\Phi\,g^\bot$, using $A_\xi=0$ and $\nabla_\xi\, g^\bot =0$ (for Riemannian foliations), we get \eqref{E1-b-T-Riem}.
Derivation of \eqref{E-genRicN-W}$_1$ in $t$ and using \eqref{E1-b-T-Riem} yield \eqref{E-RbotT-g2}.
\end{proof}

\begin{example}\rm
Let $\calf$ be a one-dimensional $\mathfrak{g}$-foliation by geodesics spanned by a unit vector field~$\xi$ (e.g., $\xi$ is a unit Killing vector). Then $\Ric^\bot=R_\xi$ (where $R_{\xi}: X\to R(\xi, X)\xi$ is the Jacobi operator in the $\xi$-direction), $\nabla_{\xi}R_{\xi}=0$ and \eqref{E-RbotT-g2} reads as
\begin{equation*}
 \dt\Ric^\bot = -2\,R_\xi(R_\xi + \Phi\id^\bot).
\end{equation*}
\end{example}

The following theorem shows that metrics of $\mathfrak{g}$-foliations with certain conditions can be deformed to metrics of the same type but with leafwise constant partial Ricci curvature.

\begin{theorem}
Let $\calf$ be a $\mathfrak{g}$-foliation of $(M, g_0)$ spanned by $p\ge1$ orthonormal vector fields~$\{\xi_i\}$.
If~$r_{g_0}>0$ on the orthogonal distribution ${\cal D}$, then \eqref{E-GF-Rmix-Phi} with $\Phi>0$ has a~unique solution~$g_t\ (t\in\RR)$.
Moreover,

{\rm (i)} if the following recurrent relation holds for some $\lambda: M\to\RR$ $($and $t=0):$
\begin{equation}\label{E-lambda}
 \nabla_\xi\Ric^\bot=\xi(\lambda)(\Ric^\bot - (1/p)\,\|T\|^2\id^\bot),\quad \xi\in\widetilde{\cal D},
\end{equation}
then there exists $\lim\limits_{\,t\to-\infty} \Ric^\bot(g_t)=(\Phi-(\Delta_\calf\,\lambda+|\nabla^\calf\lambda|^2)/4)\id^\bot$.

{\rm (ii)} if $\,\Ric^\bot=\mu\id^\bot$ for some leafwise constant positive function $\mu: M\to\RR$,
then
$g_t$ converges exponentially fast, as $t\to-\infty$, to a limit metric $\hat g$ with $r_{\hat g}=\Phi\,\hat g^\bot$.
\end{theorem}

\begin{proof}
For a bundle-like metric, $\calf$ is Riemannian; thus $A_i=0$ and $C_i=T^\sharp_i$.
By conditions,
\[
 0<\mu_{\rm min}\id^\bot \le \Ric^\bot \le \mu_{\rm max}\id^\bot,
\]
where $\mu_{\rm min}$ ($\mu_{\rm max}$) is minimal (maximal) eigenvalue of $\Ric^\bot$.
We~then have
\[
 0<\mu_{\rm min}\Ric^\bot \le -\sum\nolimits_i T_{\xi_i}^\sharp \Ric^\bot T_{\xi_i}^\sharp \le \mu_{\rm max}\Ric^\bot.
\]
Thus \eqref{E-RbotT-g2} yields the following differential inequalities:
\begin{equation*}
 2\Ric^\bot(\Ric^\bot +(\mu_{\rm min}-2\Phi)\id^\bot)
 \le\dt\Ric^\bot \le 2\Ric^\bot(\Ric^\bot +(\mu_{\rm max} -2\Phi)\id^\bot).
\end{equation*}
Consider the comparison matrix ODE with $2\Phi>\alpha\in\RR$,
\begin{equation}\label{E1-e-mu-Riem2}
 \dt Z = 2 Z(Z +(\alpha-2\Phi)\id^\bot).
\end{equation}
Let $\mu_i(t)$ be the eigenvalue and $e_i(t)$ the $g_t$-unit eigenvector of the solution
of \eqref{E1-e-mu-Riem2}.
Observe that \eqref{E1-e-mu-Riem2} preserves the directions of $\{e_i\}$ and yields the system
\[
 \dot\mu_i = 2\mu_i(\mu_i+\alpha-2\Phi),\quad 1\le i\le n.
\]
It has global solution
\[
 \mu_i(t)= \frac{\mu_i(0)(2\Phi-\alpha)}{\mu_i(0)+\exp(4\Phi\,t)(2\Phi-\alpha-\mu_i(0))},\quad \mu_i(0)>0.
\]
 Moreover, $\lim\limits_{t\to-\infty}\mu_i(t)=2\Phi-\alpha>0$.
By the above, \eqref{E-RbotT-g2} has a global solution $\Ric^\bot(t)$.
Thus, \eqref{E-GF-Rmix-Phi} has a global solution $g_t\ (t\in\RR)$.

(i) Notice that \eqref{E-lambda} is compatible with \eqref{E-AT-g}: RHS of equations have zero traces, and
\[
 \Delta_\calf\Ric^\bot=(4\Psi_1+\Phi)\Ric^\bot - \Psi_2\id^\bot,
\]
where $\Psi_1:=\frac14\,(\Delta_\calf\,\lambda+\|\nabla^\calf\lambda\|^2)-\Phi$ and
$\Psi_2:=\frac1p\,\|\nabla^\calf\lambda\|^2\,\|T\|^2\ge0$.
Hence each eigenvalue $\mu_j$ of $\Ric^\bot$ satisfies ODE
\[
 \dot\mu_j=4\mu_j(\mu_j+\Psi_1)-\Psi_2,
\]
which has two stationary solutions $\mu_\pm=\frac12\,(-\Psi_1\pm\sqrt{\Psi_1^2+\Psi_2})$ (functions on $M$).
Here, $\mu_+>0$ is attractor for $t\to-\infty$.
Since we assume $\Ric^\bot>0$, by the above, $\lim\limits_{\,t\to-\infty} \Ric^\bot(g_t)=\mu_+ \id^\bot$.

\smallskip

(ii) If $\Ric^\bot=\mu\id^\bot$ then $\Ric^\bot(t)=\mu(t)\id^\bot$ for all $t$, where
$\dot\mu=4\mu(\mu-\Phi)$. Hence,
\[
 \mu(t)= \frac{\mu(0)\Phi}{\mu(0)+\exp(4\Phi t)(\Phi-\mu(0))}
\]
with $\mu(0)>0$ and $\lim\limits_{t\to-\infty}\mu(t)=\Phi$.
Let $\{e_i(t)\}$ be a $g_t$-orthonormal frame of ${\cal D}$.
We then have
 $\dt e_i = (\mu - \Phi) e_i$.
Since $e_i(t)=z(t) e_i(0)$ with $z(0)=1$, then $\dt \log z(t) = \mu(t) - \Phi$.
By the~above, $z(t) = (\mu(t)/\mu(0))^{1/4}$, and
\[
 g_t(e_i(0),e_j(0))= z^{-2}(t)\,g_t(e_i(t),e_j(t)) = \delta_{ij}(\mu(0)/(\mu(t)))^{1/2}.
\]
As $t\to-\infty$, $g_t$ converges to the metric $\hat g$ determined by
$\hat g(e_i(0),e_j(0))=\delta_{ij}\sqrt{\mu(0)/\Phi}$.
\end{proof}


\baselineskip=13.3pt

\begin{thebibliography}{999.}%

\vskip-0mm\bibitem{A} D. Alekseevsky and P. Michor, {Differential geometry of $\mathfrak{g}$-manifolds}, Differential Geom. Appl. {5} (1995), 371--403

\vskip-0mm\bibitem{ah}
B. Andrews and C. Hopper, \textit{The Ricci Flow in Riemannian Geometry},
Springer, 2011

\vskip-0mm\bibitem{bhv}
L. Bedulli, W. He and L. Vezzoni,
Second-Order Geometric Flows on Foliated Manifolds, J. Geom. Anal. 28 (2018), 697--725

\vskip-0mm\bibitem{blaga}
A.\,M. Blaga, An isoparametric function on almost $k$-contact manifolds.
An. St. Univ. Ovidius Contanca 17(1) (2009), 15--22

\vskip-0mm\bibitem{b2010}
D. Blair, \textit{Riemannian Geometry of Contact and Symplectic Manifolds}, Springer, 2010

\vskip-0mm\bibitem{b1970}
D. Blair, Geometry of manifolds with structural group $U(n)\times O(s)$,
J. Diff. Geom. 4 (1970), 155--167

\vskip-0mm\bibitem{ca-toh}
G. Cairns, {A general description of totally geodesic foliations},
Tohoku Math. J. {38} (1986), 37--55

\vskip-0mm\bibitem{diT06}
L. Di Terlizzi, Scalar and ${\varphi}$-sectional curvature of a certain type of metric $f$-structures.
Mediterr. J. Math. 3, no. 3--4 (2006), 533--547

\vskip-0mm\bibitem{fp-m}
L.M. Fern\'{a}ndez and A. Prieto-Martin, On $\eta$-Einstein para-$S$-manifolds,
Bull. Malays. Math. Sci. Soc. 40 (2017), 1623--1637

\vskip-0mm\bibitem{gy}
S.I. Goldberg and K. Yano, On normal globally framed $f$-manifolds, Tohoku Math. J. 22 (1970), 362--370

\vskip-0mm\bibitem{KN} S. Kobayashi and K. Nomizu, \textit{Foundations of Differential Geometry}, v. 1,  John Wiley, 1963

\vskip-0mm\bibitem{Na1966}
H. Nakagawa, $f$-structures induced on submanifolds in spaces, almost Hermitian or
Kaehlerian, Kodai Math. Semin. Rep. 18 (1966), 161--183

\vskip-0mm\bibitem{od}
E. \"{O}z\"{u}sa\u{g}lam and E. Dikici,  Pseudo $f$-manifolds with complemented frames,
Adv. Appl. Clifford Algebr. 26, no. 1 (2016), 305--314

\vskip-0mm\bibitem{olv} P.\,J. Olver, \textit{Applications of Lie Groups to Differential Equations}.
Graduate Texts in Math. vol.~{107}, Springer-Verlag, New York, 1993.

\vskip-0mm\bibitem{r2010}
V. Rovenski, On solutions to equations with partial Ricci curvature, {J. Geom. and Physics}, 86, (2014), 370--382

\vskip-0mm\bibitem{r4}
V. Rovenski, Extrinsic geometric flows on codimension-one foliations,
J. of Geom. Analy\-sis 23(3) (2013), 1530--1558

\vskip-0mm\bibitem{rov-5}
V. Rovenski, {Prescribing the mixed scalar curvature of a foliation}.
Balkan J. of Geometry and Its Applications, Vol. 24, No. 1, (2019), 73--92

\vskip-0mm\bibitem{rov-2}
V. Rovenski, {Problems of Extrinsic Geometry of foliations}. Results of Science and Technology:
Mo\-dern mathematics and its applications. Thematic reviews, Vol. 999 (2019), 1--10

\vskip-0mm\bibitem{r24}
V. Rovenski, The partial Ricci flow for foliations, pp. 125--155,
in ``\textit{Geometry and its Applications}", Springer Proc. in Math. and Statistics {72}, Springer, New-York, 2014

\vskip-0mm\bibitem{RWo-1}
V. Rovenski and R. Wolak, Deforming metrics of foliations, Central European J. Math. 11(6) (2013), 1039--1055

\vskip-0mm\bibitem{tar}
A. Tarrio, On certain classes of metric para-$\phi$-manifolds with parallelizable kernel,
Tensor, N.S. 57 (1996), 258--267

\vskip-0mm\bibitem{Wo-comp} R. Wolak, {Foliations admitting transverse systems of differential equations},
Compositio Math. {67}, no. 1 (1988), 89--101

\vskip-0mm\bibitem{yan}
K. Yano, On a structure $f$ satisfying $f+f^3=0$, Technical Report No. 12, University of Washington, 1961

\end{thebibliography}
\end{document}